\documentclass[11pt]{article}   
\usepackage{amssymb,amscd,latexsym}   
\usepackage{amsmath}
\usepackage{amsthm}
\usepackage{mathdots}
\usepackage[pagebackref]{hyperref}
\usepackage{color}
\usepackage[all]{xy}
\usepackage{imakeidx}         
\usepackage{graphicx}        
\usepackage{graphicx,color} 
\usepackage{tikz}

\usepackage{graphicx,color} 
\usepackage{tikz}
\newcommand{\rar}{\rightarrow}
\newcommand{\lar}{\longrightarrow}
\newcommand{\llar}{-\kern-5pt-\kern-5pt\longrightarrow}

\newtheorem{Theorem}{Theorem}[section]
\newtheorem{Lemma}[Theorem]{Lemma}

\newtheorem{Proposition}[Theorem]{Proposition}
\newtheorem{Remark}[Theorem]{Remark}
\newtheorem{Example}[Theorem]{Example}

\newtheorem{Conjecture}[Theorem]{Conjecture}
\newtheorem{Definition}[Theorem]{Definition}

\newtheorem{Problem}[Theorem]{Problem}
\newtheorem{Setup}[Theorem]{Setup}
\def\sqr#1#2{{\vcenter{\hrule height.#2pt
			\hbox{\vrule width.#2pt height#1pt \kern#1pt
				\vrule width.#2pt}
			\hrule height.#2pt}}}
\def\phi{\varphi}

\DeclareMathOperator{\coker}{Coker}

\DeclareMathOperator{\rank}{rank}
\DeclareMathOperator{\Ht}{ht}

\def\xx{{\bf x}}

\def\ff{{\bf f}}

\def\ff{{\bf f}}

\def\fm{{\mathfrak m}}

\def\Ht{{\rm ht}\,}

\def\codim{{\rm codim}\,}

\def\ker{{\rm ker}\,}

\def\syz{\mbox{\rm Syz}}

\def\restr{{\kern-1pt\restriction\kern-1pt}}

\def\pp{{\mathbb P}}

\begin{document}
	\begin{center}
		{\large}{\bf\sc The structure of almost Cohen-Macaulay $3$-generated ideals of
			codimension $2$ in terms of matrix theory}
		\footnotetext{AMS Mathematics
			Subject Classification (2010   Revision). Primary 13A02, 13A30, 13D02, 13H10, 13H15; Secondary  14E05, 14M07, 14M10,  14M12.} 
		\footnotetext{	{\em Key Words and Phrases}: linear matrices, determinantal polynomials, Hessian, Hankel matrices, catalecticants, plane reduced points,  perfect ideal of codimension two, Rees algebra, associated graded ring.}

		\bigskip
		
		{\sc Ricardo Burity}\footnote{Partially supported by a CNPq grant (408698/2023-3)}\quad{\sc Thiago Fiel}\footnote{Under a CNPq grant (171302/2023-0)}\quad
		{\sc Zaqueu Ramos}\footnote{Partially supported by a CNPq grant (304122/2022-0)}  \quad
		{\large\sc Aron  Simis}\footnote{Partially
			supported by a CNPq grant (301131/2019-8).}

	\end{center}
	
	
	
	\begin{abstract}
	Let $R$ be a standard graded polynomial ring over a field $k$. 
	The paper focuses on homogeneous ideals $J \subset R$ of codimension $2$ generated by three forms of the same degree $d \geq 2$ that are almost Cohen--Macaulay, i.e., of homological dimension $2$.  Based on the structure of the minimal graded free resolution of $J$ and  numerical data encoded in certain \emph{latent data}, one introduces the notion of  \emph{level matrices} associated with these data. The main result provides a complete characterization of an almost Cohen--Macaulay $3$-generated ideal $J$ of codimension $2$ in terms of the existence of a related level matrix for which $J$ arises as the ideal of its maximal minors that fix a submatrix. One provides algebraic and geometric examples illustrating the results.	
	\end{abstract}
	
	\section{Introduction}

	Given a standard graded polynomial ring $R$ over a field $k$, the main finding of this work is the characterization of an almost Cohen--Macaulay $3$-generated ideal $J\subset R$ of codimension $2$ as the ideal of maximal minors of a suitable matrix fixing a submatrix.
	
	\subsection{Nature of the problem}


	The basic aim of the paper is at the class of homogeneous ideals $J\subset R$ of codimension two generated by three forms. 
While the main bulk of the classical literature has focused on codimension two ideals which happen to be perfect, not even the class of the $3$-generated such ideals seems to be fully discussed (see, e.g., \cite[Section 2]{PPPideals}).
At another end, arbitrary $3$-generated ideals have often been contemplated as a relevant class (see, e.g., \cite{Bruns},  \cite{Gulliksen}, \cite{Kohn}, \cite{McCu}).
	
	Now, from a homological point of view, there has been interest in collecting detailed information on such ideals $J$ of large homological dimension (see, e.g., \cite{Burch}, \cite{Kohn}, \cite{McCu2011}).
	Here, one focuses on the next best environment,  leading to the case of a three-generated codimension two non-perfect ideal  -- so to say, the ``first'' non Cohen-Macaulay case.
	Precisely, one gets a grip on almost Cohen--Macaulay $3$-generated ideals of codimension $2$, hence of  homological dimension $2$.
	Despite apparently a  tiny class, these ideals have been pursued quite thoroughly in past and recent literature, in a variety of styles, often in the case where they are Jacobian ideals of forms (\cite{Miro-et-al}, \cite{Ellia1998}, \cite{RTY}, \cite{HS}, \cite{JNS}, \cite{Ma_et_al2021}). The present work  hopefully sheds a new light on various aspects of this landscape. 
	Possibly, a watershed between this work and some of the earliest ones is that it pursuits ideal theory outcome off a thorough examination of pivotal  matrices related to the chain maps in the relevant free resolutions.
	
	More generally, in the case of any homogeneous ideal $I$ in a standard graded polynomial ring $R=k[x_1,\ldots,x_n]$, which is equigenerated in degree $d\geq 1$, an interesting problem that has been along for quite a while back is the search of upper bounds for the degrees of minimal generators of the syzygies of $I$.
	In this regard, already in \cite[Definition 2.1]{SimisOrdinary} the notion of {\em non-degeneracy} was introduced, to mean that the degrees of minimal syzygy generators of $I$ are bounded above by $d$.
	The encouragement for introducing this notion came from a result in \cite[Proposition 6]{Lazard} to the effect that, for $d\leq 2$, every such ideal is non-degenerate.
	In  \cite[Definition 2.1]{SimisOrdinary}  a characterization of a nodal cubic in $\pp^2$ was given in terms of non-degeneracy of its gradient ideal $I\subset k[x,y,z]$.
	An easy outcome of the focus here extends this result to higher nodality.
	
A strongly homological minded approach was taken up in \cite{HS} in the case of a three-dimensional ground polynomial ring. On its own right, some of the basic results obtained thereof were subsequently applied in some of the many findings of Dimca and Sticlaru (e.g., \cite{Dimca-Sticlaru2019}, \cite{Dimca_Sticlaru-Geom.Ded}, \cite{Dimca-SticlaruJacobianSameDeg}, \cite{Dimca-SticlaruJacobianDeriv}). 	
	One of the side results emerging essentially from \cite{HS} above is a criterium of non-perfectness for an equigenerated ideal in dimension $3$ generated by three forms having codimension $2$, in terms of the shifts in the first syzygy matrix of the ideal.
	The simplest proof of this criterium is possibly found in \cite[Lemma 1.1]{SimisTohaneanu}. 
	
	Roughly stated, the present work extends both
	 approaches of \cite{HS} and \cite{Dimca_Sticlaru-Geom.Ded},  the first of which was restricted to ideals in $k[x,y,z]$, the second to the Jacobian ideals of plane curves.
	 The watershed, if any, between these two landscapes as yet escapes a full understanding.
	 Examples abound both in the setup of plane curves as in  the case of non-geometric territory. The present results show that both landscapes  can be understood in terms of certain matrices,  whose role is explained through the notion of maximal minors of a matrix that fix a convenient submatrix (see, e.g., \cite{AnSi1981}, \cite{AnSi1986}).
	 
	 \subsection{The main theorem}
	 
	 In order to state the main finding of the paper, introduce the following:
	 
	 \begin{Setup}\label{all_setup}{\rm The main ingredients throughout are$:$
	 	\begin{itemize}
	 		\item [{\bf (L)}]
	 		A set of integers $d\geq 1,$ $m\geq 3,$ $1\leq\delta_1\leq \cdots\leq\delta_m $ and $\epsilon_1,\ldots,\epsilon_{m-2}\geq 1$ satisfying the following conditions: (i) $	\delta_{3}+\epsilon_{1}\leq \cdots\leq \delta_m+\epsilon_{m-2}$, (ii) $\delta_1+\delta_2=d+\sum_{j=1}^{m-2}\epsilon_j$,\, (iii)
	 		$\delta_i+\delta_j\geq {d+1} \,\,\mbox{for every}\,\,1\leq i<j\leq m$,\, (iv)
	 		$\delta_3\leq d.$
	
	 		Such an integer set will be encoded in the notation $(d,m,\underline{\delta}, \underline{\epsilon})$ and  refer to as {\em latent data}.
	 		
	 		
	 		\item [{\bf (M)}] Given latent data $(d,m,\underline{\delta}, \underline{\epsilon})$, an $(m+1)\times m$ vertical block matrix $\eta:=\left[\begin{array}{c} 
	 			A \\ 
	 			\hline 
	 			B \end{array}\right],$ such that:
	 		\begin{enumerate}
	 			\item[{\rm (a)}] $A:=(a_{i,j})$ is a $3\times m$ matrix with entries in $R$ such that $a_{i,j}$ is a homogeneous polynomial of degree $d-\delta_j$ if  $d-\delta_j\geq 0,$ and   $a_{i,j}=0$ if $d-\delta_j< 0$.
	 			\item[{\rm (b)}]  $B:=(b_{i,j})$ is an $(m-2)\times m$ matrix with entries in $R$ such that $b_{i,j}$ is a homogeneous polynomial of degree $\delta_{i+2}-\delta_j+\epsilon_i$ if $\delta_{i+2}-\delta_j+\epsilon_i>0$ and $b_{i,j}=0$ if $\delta_{i+2}-\delta_j+\epsilon_i\leq 0.$ 
	 			\item[{\rm (c)}] ${\rm ht\,}I_m(\eta)=2$, and ${\rm ht\,}I_{m-2}( B)=3,$ where ${\rm ht\,}$ denotes \emph{height} (codimension).
	 		\end{enumerate}
	 	\end{itemize}
	 }
	 \end{Setup}

	 \begin{Definition}\label{def_level-first}\rm
	 	A matrix such as $\eta$ above
	 	will be said to be 
	 	{\em $(d,m,\underline{\delta}, \underline{\epsilon})$-level}.
	 \end{Definition}
	 
	 The main result of the paper is the following theorem.
	 
	 \begin{Theorem}\label{mainthm}  Let $J\subset R$ be an ideal.  The following conditions are equivalent$:$
	 	\begin{itemize}
	 		\item[{\rm (i)}]  $J$ is an almost Cohen--Macaulay codimension $2$ ideal generated by three forms of the same degree $d\geq 2$.
	 		\item[{\rm (ii)}] There exist  latent data $(d,m, \underline{\delta},\underline{\epsilon})$, and a $(d,m, \underline{\delta},\underline{\epsilon})$-level  matrix $\eta$ such that $J$ is generated by the maximal minors of $\eta$ fixing its lower block consisting of $m-2$ rows.
	 	\end{itemize}
	 \end{Theorem}
	 
	 As it happens, none of the two implications is obvious.
	 
	 One now briefly discusses the main results of the paper.
	 
	 Risking repetition, throughout $R=k[x_1,\ldots,x_n]$ denotes a standard graded polynomial ring over a field $k$, and $d\geq 1$ is an integer.
	 As explained above, the basic piece is an  ideal $J\subset R$  generated by three forms of degree $d$, assumed to be of codimension $2$ and homological dimension $2$.
	 
	 One cornerstone  is the notion of  latent data as introduced above, by establishing that they naturally emerge out of the shape of the generating syzygies of the minimal free resolution of $J$.
	 Some of these are granted by findings of \cite{HS} in the plane (i.e., over $k[x,y,z]$) easily converted to $R$ as in here (cf. \eqref{bigsumHS} and \eqref{sumHS}). The remaining property of these shifts is obtained by extending to  arbitrary forms a result of Dimca--Sticlaru on the partial derivatives of a plane curve (Theorem~\ref{D-Sextended}).
	 
	 A second cornerstone is the  notion of a level matrix in $2$-block format, based on a set of latent data, as defined in detail above.
	 
	 Bound together, these two ideas frame the standing findings of the paper.
	 Sided with a version of the classical notion of ideals of minors of a matrix fixing a submatrix, they furnish the main propositions in this work, namely,  Proposition~\ref{AS-implies-ideal-prescribed-shifts} , Proposition~\ref{pre-level}, and Theorem~\ref{mainthm}.
	 The first of these results explains how a level matrix $\eta$ based on latent data, with the concurrence of a certain skew-symmetric matrix $K$ of rank $2$, produces the minimal free resolution of the three maximal minors of $\eta$ fixing its lower block.
	 
	 The role of  $K$ in  the context is in row with an earlier idea of Vasconcelos, and together they are better recovered in terms of compound matrices, an idea one explores to some extent  in order to pull out the elements of Proposition~\ref{pre-level}.
	 
	Altogether, these two propositions lead to the statemente and proof of Theorem~\ref{mainthm}.
	 
	 The last part of the paper consists of selected examples, both geometric and non-geometric, illustrating the role of the main results.
	 Some of these examples have been mentioned before in the literature, possibly in a different context.

	\section{On the shifts of the minimal free resolution}
	
	\subsection{A syzygy uperbound of Dimca--Sticlaru in the case of forms}\label{D-S}
	
	Let $R=k[x_1,\ldots,x_n]$ denote a standard graded polynomial ring in $n\geq3$ variables over a field $k$. If $J$ stands for a homogeneous ideal of $R$ generated by three forms of degree $d$, of homological dimension $\rho\geq 2$, then, its minimal graded free resolution has the form 
	
	\begin{equation}\label{res-three-generated}
		0\to F_{\rho}\to \cdots \to F_2 \to \bigoplus_{i=1}^{m} R(-d-\delta_i)\stackrel{\phi}\to R(-d)^3\to R,
	\end{equation}
	for suitable $m\geq 3$ and shifts $\delta_1\leq \delta_2\leq \delta_3\leq\cdots\leq \delta_m,$ where $F_2,\ldots,F_{\rho}$ are suitable graded free modules.

	
	A preliminary result in the search of bounds of the shifts $\delta_i$ above is the following:
	\begin{Proposition}\label{cod2_3gens}
		Let $R=k[x_1,\ldots,x_n]$ be a standard graded polynomial ring in $n\geq 3$ variables over a field $k$.  If $J$ is a homogeneous ideal of $R$ of height $2$ generated by three forms of degree $d$, then the following conditions are equivalent$:$
		\begin{itemize}
			\item[{\rm (i)}] $J$ is not a perfect ideal.
			\item[{\rm (ii)}] For every  two distinct minimal generating syzygies  the sum of their degrees is at least $d+1$.
		\end{itemize}
	\end{Proposition}
	The main implication of this proposition, namely, (i) $\Rightarrow$ (ii), is proved in \cite[Lemma 1.1 (iii)]{SimisTohaneanu} assuming that $n=3$, but a close scrutiny of the details shows that $n$ can be arbitrarily $\geq 3$.
	
	The following theorem extends to three arbitrary forms in $R=k[x_1,\ldots,x_n]$ ($n\geq 3$)  the result of \cite[Theorem 2.4]{Dimca_Sticlaru-Geom.Ded}, the latter in the case of the partial derivatives of a reduced plane curve. The present argument draws on the one in \cite[Theorem 2.4]{Dimca_Sticlaru-Geom.Ded} with due care to adjustments.  The use of Proposition~\ref{cod2_3gens} is pivotal.
	
	\begin{Theorem}\label{D-Sextended}
		Let $R=k[x_1,\ldots,x_n]$ be a standard graded polynomial ring in $n\geq 3$ variables over a field $k$.  Let $J$ stand for a non-perfect homogeneous ideal of $R$ of height $2$ generated by three forms of degree $d$. Letting  {\rm (\ref{res-three-generated})} above stand for its minimal graded free resolution, then $\delta_3\leq d$.
	\end{Theorem}
	\begin{proof} 
		Say, $J=\langle f_1,f_2,f_3\rangle$,  assuming as we may that $\{f_1,f_2\}$ is a regular sequence.
		Let 
		$${\rm Syz}(J):=\phi \left(\bigoplus_{i=1}^{m} R(-d-\delta_i)\right)=\ker \left(R(-d)^3\to R\right)$$ 
		stand for the module of first syzygies of $J$ with respect to the set of generators $f_1,f_2,f_3.$ 
		
		{\sc Claim.}  If $\mathfrak{z}_0\in {\rm Syz}(J)$ is a nonzero syzygy of minimal degree ($=\delta_1$) then the $R$-module ${\rm Syz}(J)/R\mathfrak{z}_0$ is torsion free.

		The proof of the claim will consist in showing that ${\rm Syz}(J)/R\mathfrak{z}_0$ embeds as a submodule of the free module $\bigwedge^2R^3$ by exhibiting a nonzero $R$-modulo map $v:{\rm Syz}(J)\rightarrow \bigwedge^2R^3$ whose kernel is generated by $\mathfrak{z}_0$.
		The required map will be the restriction map of the nonzero $R$-modulo map $V: R^3\rightarrow  \bigwedge^2R^3$, defined as follows by means of its effect on the canonical basis
		${\bf e_1}=	[\begin{matrix}1&0&0\end{matrix}]^t, {\bf e_2}=	[\begin{matrix}0&1&0\end{matrix}]^t, {\bf e_3}=	[\begin{matrix}0&0&1\end{matrix}]^t:$
		$$V(\bf e_i):=\mathfrak{z}_0\wedge e_i.$$
		
		Now, write ${\mathfrak z}_0=(p_1 \ p_2 \ p_3)^t$, as an element of ${\rm Syz}(J)\subset R^3$.
		Let ${\mathfrak z}=(q_1 \ q_2 \ q_3)^t\in{\rm Syz}(J)$ be arbitrary.

		Then 
		$$v(\mathfrak{z})=\det\left[\begin{matrix}p_1&q_1\\p_2&q_2\end{matrix}\right]{\bf e}_1\wedge{\bf e}_2+\det\left[\begin{matrix}p_2&q_2\\p_3&q_3\end{matrix}\right]{\bf e}_1\wedge{\bf e}_3+\det\left[\begin{matrix}p_1&q_1\\p_3&q_3\end{matrix}\right]{\bf e}_2\wedge{\bf e}_3.$$
		Hence, $\mathfrak{z}\in\ker v$ if, and only if, the rank of the matrix 
		$$\left[
		\begin{array}{cc}
			p_1 & q_1 \\
			p_2 & q_2 \\
			p_3 & q_3 
		\end{array}
		\right]
		$$
		is one. In particular, in this case, there are  nonzero  coprime $r,s\in R$ such that $rp_i=sq_i$,
		for every $1\leq i\leq 3.$ Since $\mathfrak{z}_0$ is a syzygy of minimal degree, $\langle p_1,p_2,p_3\rangle$ has height at least $2$.
		Therefore,  $s$ is a nonzero element of $k.$ Thus, $\mathfrak{z}\in R\mathfrak{z}_0$ as was to be shown.
		
		To proceed with the proof of the main statement, suppose by way of contradiction that  $\delta_3>d.$
		Then, due to the Koszul syzygies
		which live in degree $d$, and the assumption that  $\delta_1,\delta_2$ are degrees of minimal generators, we know that $\delta_1,\delta_2\leq d.$ Moreover, because a next minimal degree of a generator is assumed to be greater than $d$,  the Koszul syzygies $\mathbf{k}_{i,j}$ can be written in terms of the syzygy $\mathfrak{z}_0$ of degree $\delta_1$, and a syzygy $\mathfrak{z}_1$ of  degree  $\delta_2$, say$:$
		$$\mathbf{k}_{2,3}=a_{2,3}\mathfrak{z}_0+b_{2,3}\mathfrak{z}_1,\quad \mathbf{k}_{1,3}=a_{1,3}\mathfrak{z}_0+b_{1,3}\mathfrak{z}_1,\quad \mathbf{k}_{1,2}=a_{1,2}\mathfrak{z}_0+b_{1,2}\mathfrak{z}_1$$
		for some homogeneous polynomials $a_{2,3},a_{1,3}, a_{1,2}\in R_{d-\delta_1}$ and $b_{2,3},b_{1,3}, b_{1,2}\in R_{d-\delta_2}.$ 
		
		Confronting with the relation
		$$f_1\mathbf{k}_{2,3}+f_2\mathbf{k}_{1,3}+f_3\mathbf{k}_{1,2}=0$$
		yields
		\begin{equation}\label{torsion}
			\alpha\mathfrak{z}_0+\beta\mathfrak{z}_1=0,
		\end{equation}
		where
		$$\alpha=a_{2,3}f_1+a_{1,3}f_2 +a_{1,2}f_3\quad \mbox{and}\quad \beta=b_{2,3}f_1+b_{1,3}f_2 +b_{1,2}f_3.$$
		
		If $\beta\neq 0$ then \ref{torsion} implies that the residual class of $\mathfrak{z}_1$ in ${\rm Syz}(J)/R\mathfrak{z}_0$ is a non-zero torsion element, contradicting the above claim.
		
		Now, since $J$ is not perfect, by Proposition~\ref{cod2_3gens}, $d-\delta_2\leq \delta_1-1.$
		Thus, if at the other end, $\beta=b_{2,3}f_1+b_{1,3}f_2 +b_{1,2}f_3=0$ then $b_{2,3}=b_{1,3}=b_{1,2}=0$ because the minimal degree of a nonzero syzygy is $\delta_1.$ 
		In particular, $\mathbf{k}_{2,3}=a_{2,3}\mathfrak{z}_0$ and $\mathbf{k}_{1,3}=a_{1,3}\mathfrak{z}_0$, and hence $\{f_1,f_2\}\subset \langle p_3\rangle$,
		a contradiction since $\{f_1,f_2\}$ is a regular sequence.
	\end{proof}

	In this paper we are interested in the case where $J$ stands for a homogeneous ideal of $R$ generated by three forms of degree $d$, of homological dimension exactly $2.$ In particular, the minimal graded free resolution of $J$ takes the form
	\begin{equation}\label{res-prelim}
		0\to \bigoplus_{j=1}^{m-2} R(-D_j)\to \bigoplus_{i=1}^{m} R(-d-\delta_i)\to R(-d)^3\to R,
	\end{equation}
	where $m\geq 3$, for suitable shifts $\delta_1\leq \delta_2\leq \delta_3\leq\cdots\leq \delta_m,$ and $D_1\leq D_2\leq \cdots\leq D_{m-2}.$
	According to \cite[Lemma 1.1]{HS},
	\begin{equation}\label{bigsumHS}
		D_j=d+\delta_{j+2}+\epsilon_j \quad 1\leq j\leq m-2
	\end{equation}
	for certain positive integers $\epsilon_j\geq 1.$  
	We can also deduce, similarly to \cite[Formula (13)]{HS}, and regardless of the dimension $\geq 3$ of $R$, that  
	\begin{equation}\label{sumHS}
		\delta_1+\delta_2=d+\sum_{j=1}^{m-2}\epsilon_j.
	\end{equation} 
	Namely, drawing upon the free resolution \eqref{res-prelim},  the Hilbert series of $R/I$ is
	\begin{equation}\label{HSerie}
		\frac{1-3t^d+\sum_{i=1}^{m}t^{d+\delta_i}-\sum_{j=1}^{m-2}t^{d+\delta_{j+2}+\epsilon_j}}{(1-t)^n}
	\end{equation}
	Taking $t$-derivatives of the numerator of \eqref{HSerie}
	evaluated at $t = 1$ (see \cite[Corollary 4.1.14]{BHbook}), one obtains the desired relation.

\section{Level matrices and nature of shifts}

	Let $R=k[x_1,\ldots,x_n]$ denote a standard graded polynomial ring in $n\geq3$ variables over a field $k$ -- to be fixed throughout unless explicitly stated. As seen in Subsection~\ref{D-S}, if $J$ is a homogeneous ideal of $R$ of height 2 generated by three forms of degree $d$, with homological dimension $2$, then, its minimal graded free resolution has the form
\begin{equation}\label{res-prelim-bis}
	0\to \bigoplus_{j=1}^{m-2} R(-d-\delta_{j+2}-\epsilon_j)\stackrel{\psi}\to \bigoplus_{i=1}^{m} R(-d-\delta_i) \stackrel{\varphi}\to R(-d)^3\to R ,
\end{equation}
where $m\geq 3,$  $\delta_1\leq \delta_2\leq \delta_3\leq\cdots\leq \delta_m,$ and $\epsilon_1,\ldots,\epsilon_{m-2}$ are positive integers satisfying the following conditions

\begin{equation}\label{crescente}
		\delta_{3}+\epsilon_{1}\leq \cdots\leq \delta_m+\epsilon_{m-2} \quad (\mbox{from Formula~\eqref{bigsumHS}})
\end{equation}
\begin{equation}\label{sumHS-c1}
	\delta_1+\delta_2=d+\sum_{j=1}^{m-2}\epsilon_j, \quad (\mbox{as in Formula~\eqref{sumHS}})
\end{equation} 
\begin{equation}\label{c2}
	\delta_i+\delta_j\geq {d+1} \,\,\mbox{for every}\,\,1\leq i<j\leq m, \quad (\mbox{as in Proposition~\ref{cod2_3gens}})
\end{equation}
and 
\begin{equation}\label{c3}
	\delta_3\leq d. \quad (\mbox{by Theorem~\ref{D-Sextended}})
\end{equation}

A set of integers $d\geq 1,$ $m\geq 3,$ $1\leq\delta_1\leq \cdots\leq\delta_m $ and $\epsilon_1,\ldots,\epsilon_{m-2}\geq 1$ satisfying the above conditions \eqref{crescente}, \eqref{sumHS-c1}, \eqref{c2} and \eqref{c3} is what one called {\em latent data} in the Introduction. 
One keeps the same   notation $(d, m, \underline{\delta},\underline{\epsilon})$. 
One is interested in the following converse-like problem:

\begin{Problem}\label{mainquestion}\rm
	Given a set of latent data $(d, m, \underline{\delta},\underline{\epsilon})$  is there  a homogeneous ideal  $J$ of $R$ of height 2 generated by three forms of degree $d$, of homological dimension $2$, such that the minimal graded free resolution of $J$ has the form in  \eqref{res-prelim-bis}$?$
\end{Problem}

\subsection{Level matrices and ideals of maximal  minors fixing a submatrix}

 In this subsection one approaches Problem~\ref{mainquestion} by evoking an explicit method  based on the free resolutions of  ideals of maximal minors fixing  a submatrix, as approached by Andrade-Simis in \cite{AnSi1986}. 
 
Throughout,  $(d,m, \underline{\delta},\underline{\epsilon})$ denotes a set of latent data, as introduced  previously. 
\begin{Remark}\label{prop_of_latent}\rm
Note that, for $1\leq i\leq m-2$ and $1\leq j\leq m$ such that $j\leq i+2$, one has $\delta_{i+2}-\delta_j+\epsilon_i>0.$
\end{Remark}
This follows because
$ j\leq i+2$ implies $\delta_j\leq \delta_{i+2}$, since the sequence of the $\delta$'s is non-decreasing.
Hence,
$\delta_{i+2}-\delta_j\geq 0$.
Adding $\epsilon_i$ becomes positive as the latter is positive.

This observation is relevant in that it guarantees the meaning of  item $B$ in the notion of a level matrix associated to these latent data as in Setup~\ref{all_setup}, which one now reinstates for the  reader's convenience.

\begin{Definition}\label{def_level}\rm
An $(m+1)\times m$ vertical block matrix $\eta:=\left[\begin{array}{c} 
	A \\ 
	\hline 
	B \end{array}\right]$ is said to be {\em $(d,m,\underline{\delta}, \underline{\epsilon})$-level}  if it satisfies the following conditions:
	
	$\bullet$ $A:=(a_{i,j})$ is a $3\times m$ matrix with entries in $R$ such that $a_{i,j}$ is a homogeneous polynomial of degree $d-\delta_j$ if  $d-\delta_j\geq 0, $ and $a_{i,j}=0$ if $d-\delta_j< 0$.
	
$\bullet$  $B:=(b_{i,j})$ is an $(m-2)\times m$ matrix with entries in $R$ such that $b_{i,j}$ is a homogeneous polynomial of degree $\delta_{i+2}-\delta_j+\epsilon_i$ if $\delta_{i+2}-\delta_j+\epsilon_i>0$, and $b_{i,j}=0$ if $\delta_{i+2}-\delta_j+\epsilon_i\leq 0.$

$\bullet$ ${\rm ht\,}I_m(\eta)=2$, and ${\rm ht\,}I_{m-2}( B)=3$.
\end{Definition}

Matrices fulfilling the conditions of $\eta$ above are quite natural, with the shape of the lower block $B$  having been around in the literature. 
Here is one illustration.

\

\begin{Proposition}\label{exlevel}
Let $(d,m, \underline{\delta},\underline{\epsilon})$ be a set of latent data satisfying the following additional condition$:$ 
\begin{equation}\label{condition*}
		\delta_{i+2}-\delta_{j}+\epsilon_{i}>0  \quad  \mbox{for each $1\leq i\leq m-2$ and $1\leq j\leq m$  such that $j=i+3$}.
\end{equation}
Then the following $(m+1)\times m$ block matrix over $k[x,y,z]$ is a $(d,m, \underline{\delta},\underline{\epsilon})$-level matrix$:$
{\small
\begin{equation*}
	\eta:=\left[	\begin{array}{ccccccccccc}
		x^{\alpha_{1,1}}\\
		y^{\alpha_{2,1}}&x^{\alpha_{2,2}}&\\
		&y^{\alpha_{3,2}}&x^{\alpha_{3,3}}\\
		\hline\\ [-5pt]
	0&z^{\beta_{1,2}}&y^{\beta_{1,3}}&x^{\beta_{1,4}}&0&0&\cdots&0&0&0&\\
	0&0&z^{\beta_{2,3}}&y^{\beta_{2,4}}&x^{\beta_{2,5}}&0&\cdots&0&0&0&\\
	0&0&0&z^{\beta_{3,4}}&y^{\beta_{3,5}}&x^{\beta_{3,6}}&\cdots&0&0&0&\\
	\vdots & \vdots & \vdots & \vdots & \vdots & \vdots&\vdots&\vdots&\vdots&\vdots&\\
	0&0&0&0&0&0&\cdots&z^{\beta_{m-3,m-2}}&y^{\beta_{m-3,m-1}}& x^{\beta_{m-3,m}}&\\
	x^{\beta_{m-2,1}}&0&0&0&0&0&\cdots&0&z^{\beta_{m-2,m-1}}&y^{\beta_{m-2,m}}&	
	\end{array}	\right],
\end{equation*}
}
where $\alpha_{i,j}=d-\delta_j$, $\beta_{i,j}=\delta_{i+2}-\delta_j+\epsilon_i,$ and the empty slots are null entries.
\end{Proposition}
\begin{proof}
One has to show  that $\Ht I_m(\eta)\geq 2$ and $\Ht I_{m-2}(B)\geq 3$, for which one argues as follows.
	
	Let $p_1,\ldots,p_{m+1}$ be the ordered signed maximal minors of $\eta.$ 
	
	Note that
	\begin{eqnarray*}
	p_1&\equiv & y^{\alpha_{2,1}+\alpha_{32}+\beta_{13}+\cdots+\beta_{m-2m}}+x^{\beta_{m-21}+\alpha_{22}+\alpha_{33}+\beta_{1,4}+\cdots+\beta_{m-3m}} \: (\bmod z)\\
	p_{m+1} &\equiv & x^{\alpha_{11}+\alpha_{22}+\alpha_{33}+\beta_{1,4}+\cdots+\beta_{m-3m}}\: (\bmod z).
	\end{eqnarray*}
	Since the respective images of $p_1$ and $p_{m+1}$ in $k[x,y,z]/\langle z\rangle$ are coprime, then $p_1,p_{m+1}$ are coprime in $k[x,y,z].$ Thus, $\Ht I_{m}(\eta)\geq 2.$ 
	
	On the other hand, by the shape of the lower block $B$ one sees that the regular sequence $\{x,y,z\}$ is contained in any prime ideal containing $I_{m-2}(B).$ Therefore, $\Ht I_3(B)\geq 3.$

This wraps up the argument.
\end{proof}

\medskip

To endorse the relevance of condition \eqref{condition*} in Proposition~\ref{exlevel},  consider the following example.

\begin{Example}\rm Let $d=2,$ $m=4,$ $\underline{\delta}=(2,2,2,\delta_4)$  ($\delta_4\geq 3)$ and $\underline{\epsilon}=(1,1).$ 
	One can readily verify that this choice entails a latent data $(d,m,\underline{\delta},\underline{\epsilon})$. Note that $d-\delta_4\leq -1$ and $\delta_3-\delta_4+\epsilon_1\leq 0.$  Thus,  a matrix $\eta$ satisfying the first two conditions in the definition of a $(d,m,\underline{\delta},\underline{\epsilon})$-level matrix must have the following format:
	
	$$\eta=\left[\begin{matrix}
		a_{1,1}&a_{1,2}&a_{1,3}&0\\
		a_{2,1}&a_{2,2}&a_{2,3}&0\\
		a_{3,1}&a_{3,2}&a_{3,3}&0\\
		b_{1,1}&b_{1,2}&b_{1,3}&0\\
		b_{2,1}&b_{2,2}&b_{2,3}&b_{2,4}\\
	\end{matrix}\right].$$
	But, in this case, $I_{4}(\eta)\subset\langle b_{2,4}\rangle,$ hence the third defining condition breaks down. Therefore,  a $(d,m,\underline{\delta},\underline{\epsilon})$-level matrix is not available in this case.
\end{Example}
In particular, by Theorem~\ref{mainthm}, no almost Cohen--Macaulay codimension $2$ ideal generated by $3$ forms of degree $2$ exists for which the shifts of its free resolution afford the above latent data. This can be verified directly, with no appeal to Theorem~\ref{mainthm}, as the free resolution would have to be of the form
$$0\to R(-5)\oplus R(-6)\to R(-4)^3\oplus R(-5)\to R(-2)^3\to R.$$
Computing the numerator of the Hilbert series with respect to a variable $t$, taking second $t$-derivatives and evaluating at $t=1$ yields a vanishing multiplicity, which is absurd.

\medspace

The result below is somewhat inspired by the contents of \cite{AnSi1981} and \cite{AnSi1986},  giving it extra precision in the case of level matrices.

\begin{Proposition}\label{AS-implies-ideal-prescribed-shifts} 
	Let $\eta:=\left[\begin{array}{c} 
		A \\ 
		\hline 
		B \end{array}\right]$ denote a  $(d,m, \underline{\delta},\underline{\epsilon})$-level matrix over $R$. Order the  signed maximal minors of $\eta$ in such a way that $p_1,p_2$ and $p_3$ are the three maximal minors of $\eta$ fixing the lower block $B.$ Then, the minimal graded free resolution of $J:=\langle p_1,p_2,p_3\rangle\subset R$ is
	\begin{equation}
		0\to \bigoplus_{j=1}^{m-2} R(-d-\delta_{j+2}-\epsilon_j)\stackrel{B^t}\lar \bigoplus_{i=1}^{m} R(-d-\delta_i)\stackrel{AK}\lar R(-d)^3\stackrel{[p_1\,p_2\,p_3]}\lar R ,
	\end{equation}
with
\begin{equation}\label{KassociatedtoB}
K=\left[\begin{matrix}
	0&\sigma_{1,2}\Delta_{1,2}&\cdots&\sigma_{1,m}\Delta_{1,m}\\
	\sigma_{2,1}\Delta_{2,1}&0&\cdots&\sigma_{2,m}\Delta_{2,m}\\
	\vdots&\vdots&\ddots&\vdots\\
	\sigma_{m,1}\Delta_{m,1}&\sigma_{m,2}\Delta_{m,2}&\cdots&0
\end{matrix}\right],
\end{equation}
where, for every $1\leq i,j\leq m$ with $i\neq j,$ $\Delta_{i,j}$ is the maximal minor of $B$ obtained by omitting the $i$th and $j$th columns and $\sigma_{i,j}=(-1)^{i-j}.$ 
\end{Proposition}
\begin{proof} Since $\Ht I_{m-2}(B)=3,$  $\coker B$ is resolved by the Buchsbaum-Rim complex (see \cite[Corollary A2.13]{E})
	\begin{equation}\label{B-R-cokerB}
		0\to R^{m-1}\stackrel{B^t}\lar R^m\stackrel{K}\lar R^m\stackrel{B}\to R^{m-1}\to\coker B\to 0.
	\end{equation}
	Let $\{p_1,\ldots, p_{m+1}\}$ be the signed maximal minors of $\eta$, ordered as stated.
	As clearly,  $\left[\begin{matrix}p_4&\cdots&p_{m+1}\end{matrix}\right]B=-\left[\begin{matrix}p_1&p_2&p_3\end{matrix}\right]A$, one deduces from \eqref{B-R-cokerB} the following complex
	\begin{equation}
		0\to R^{m-2}\stackrel{B^t}\lar R^m\stackrel{AK}\lar R^3\stackrel{[p_1\,p_2\,p_3]}\lar R.
	\end{equation}
	Now, since $\Ht I_2(\eta)=2$ and $\Ht I_{m-2}(B)=3,$ it follows from \cite[Theorem D]{AnSi1986} that this complex
	is a free resolution of $J.$
	
	 Thus,  to conclude one has to verify the shifts.
	 Namely, it is enough to show that $\deg p_1=\deg p_2=\deg p_3$ and that the degree of the $j$th column of $AK$ is $\delta_j$ for every $1\leq j\leq m.$

First, note that
	\begin{equation}\label{format-p_i}
		p_i=\sum_{1\leq r<s\leq m} g_{r,s}^{\hat{i}}\Delta_{r,s} \quad (1\leq i\leq 3)
	\end{equation}
where $g_{r,s}^{\hat{i}}$  stands for the $2$-minor of $A$ omitting the $i$th row and choosing the $r$th and $s$th columns.  Since the $j$th column of $A$ is null if $d-\delta_j<0 $, then the nonzero summands in \eqref{format-p_i} are those such that $d-\delta_r\geq 0$ and $d-\delta_s\geq 0.$ For any such a summand, one has
	$$\deg g_{r,s}^{\hat{i}}=2d-\delta_r-\delta_s\quad \mbox{and}\quad \deg \Delta_{r,s}= \sum_{j=3}^{m}\delta_j-\sum_{1\leq j\leq m\atop j\neq r,s}\delta_j+\sum_{j=1}^{m-2} \epsilon_j ,$$
	that is,
	$$\deg(g_{r,s}^{\hat{i}}\Delta_{r,s})=2d-\delta_1-\delta_2+\sum_{j=1}^{m-2} \epsilon_j $$
	Thus, from  \eqref{sumHS-c1} it follows that
	$$\deg(g_{r,s}^{\hat{i}}\Delta_{r,s})=d.$$
	In conclusion,  $\deg p_1=\deg p_2=\deg p_3=d.$

	Now let $c_{i,j}$ be the $i,j$th entry of the matrix $AK,$ namely, 
		$$c_{i,j}=\sigma_{1,j}a_{i,1}\Delta_{1,j}+\cdots +\sigma_{j-1,j}a_{i,j-1}\Delta_{j-1,j}+\sigma_{j+1,j} a_{i,j+1}\Delta_{j+1,j}+\cdots +\sigma_{m,j}a_{i,m}\Delta_{m,j}.$$
		One needs to show that $\deg c_{i,j}:=\delta_j.$
	But since 
	$$\deg a_{i,u}=d-\delta_u\quad \mbox{and} \quad\deg \Delta_{u,j}= \displaystyle\sum_{l=3}^{m}\delta_l-\displaystyle\sum_{1\leq l\leq m\atop l\neq u,j}\delta_l+\displaystyle\sum_{l=1}^{m-2} \epsilon_l,$$ 
	\eqref{sumHS-c1} again implies that
	$\deg(a_{i,u}\Delta_{u,j})=\delta_j.$
	Thus, $\deg c_{i,j}=\delta_j.$
	\end{proof}

\subsection{Skew-symmetric matrices and an idea of Vasconcelos}

Let $\eta$ be a  $(d,m, \underline{\delta},\underline{\epsilon})$-level  matrix based on a set of latent data and let $J$ denote the ideal generated by the maximal minors of $\eta$ fixing the lower block matrix $B.$ By Theorem \ref{AS-implies-ideal-prescribed-shifts}, in order  that  $J$ provide an affirmative answer to Problem~ \ref{mainquestion} it suffices to verify that it has height $2$.  

An argument will be supplied here as based on a lemma about compound  matrices of skew-symmetric matrices of rank $2$, and  a result first guessed by W. Vasconcelos.

Let us proceed to the required details.

Let $M$ be an $m\times n$ matrix with entries in an arbitrary ring $R.$  For nonempty subsets $\alpha\subset\{1,\ldots,m\}$  and  $\beta\subset \{1,\ldots,n\}$,  $M(\alpha|\beta)$ denotes the submatrix of $M$ with rows (respectively, columns) indexed by $\alpha$ (respectively, $\beta$), lexicographically ordered. Let $p\leq \min\{m,n\}$ denote a positive integer.  The  {\em $p$-compound}   of the matrix $M$ is the ${m\choose p}\times {n\choose p}$ matrix $C_p(M)$ whose  entries are the (determinantal) minors $\det M(\alpha|\beta)$, for all choices of $\alpha\subset\{1,\ldots,m\}$ and $\beta\subset \{1,\ldots,n\}$ such that $\#\alpha=\#\beta=p$.

It is classically known that, as a consequence of the Cauchy-Binet formula's, one can infer that, for any $m\times n$ matrix $M$ and any $n\times l$ matrix $N$, one has the compound property
\begin{equation}\label{compound-property}
	C_p(MN)=C_p(M)C_p(N),
\end{equation}
for every $1\leq p\leq \min\{m,n,l\}.$

Skew-symmetric matrices of rank $\leq 2$ interact with the $2$-compound matrices, in the following sense.

\begin{Lemma}\label{compoundskwerank2}
	Let $M=(a_{i,j})$ be an $m\times m$ skew-symmetric matrix over an integral domain $R.$ Suppose that$:$ 
	\begin{enumerate}
	\item[{\rm (1)}] $M$ has rank $\leq 2.$
		\item[{\rm (2)}]  The entries off the main diagonal are nonzero.
\end{enumerate}
	Then, for every $1\leq i< j\leq m, $ the column of $C_{2}(M)$ determined by indices $\{i,j\}$ is the transpose of the following $1\times {m\choose 2}$ matrix
	$$\left[\begin{matrix}
		a_{i,j} a_{1,2}&\cdots &a_{i,j} a_{1,m}&a_{i,j} a_{2,3}&\cdots&a_{i,j} a_{2,m}&\cdots& a_{i,j} a_{m-1,m}
	\end{matrix}\right]$$
\end{Lemma}
\begin{proof} By definition, given indices $1\leq i<j\leq m,$ the entries of the column of $C_2(M)$ determined by $\{i,j\}$ are the $2$-minors of the $m\times 2$ submatrix of $M$ which is the transpose of the following matrix:
	$$N=\left[\begin{array}{ccccccccccc}
a_{1,i}&\cdots&a_{i-1,i}&0&-a_{i,i+1}&\cdots& -a_{i,j-1}&-a_{i,j}&-a_{i,j+1}&\cdots&-a_{i,m}\\ 
a_{1,j}&\cdots&a_{i-1,j}&a_{i,j}&a_{i+1,j}&\cdots& a_{j-1,j}&0&-a_{j,j+1}&\cdots&-a_{j,m}
	\end{array}\right].$$
	Given $1\leq u<v\leq m$, let $\theta_{u,v}$ denote the $2$-minor of $N$ with rows $u$ and $v$. One needs to show that $\theta_{u,v}=a_{i,j}a_{u,v}$ for every $1\leq u<v\leq m.$ 
	
	For the 2-minors of $N$ fixing the $i$th row one has:
	\begin{equation}
		\theta_{u,i}=\det\left[\begin{matrix}
			a_{u,i}&0\\
			a_{u,j}&a_{i,j}
		\end{matrix}\right]=a_{i,j}a_{u,i},\quad \mbox{for every $1\leq u\leq i,$}
	\end{equation}
	\begin{equation}
		\theta_{i,u}= \det\left[\begin{matrix}
			0&a_{i,u}\\
			a_{i,j}&\ast
		\end{matrix}\right]=a_{i,j}a_{i,u},\quad \mbox{for every $i+1\leq u\leq m$}.
	\end{equation}
	Similarly, for the $2$-minors of $N$ fixing the $j$th row of $N$ it obtains:
	\begin{equation}
		\theta_{u,i}=\det\left[\begin{matrix}
			\ast&-a_{i,j}\\
			a_{u,j}&0
		\end{matrix}\right]=a_{i,j}a_{u,j},\quad \mbox{for every $1\leq u\leq j-1,$}
	\end{equation}
	\begin{equation}
		\theta_{j,u}= \det\left[\begin{matrix}
			-a_{i,j}&a_{i,u}\\
			0&-a_{j,u}
		\end{matrix}\right]=a_{i,j}a_{j,u},\quad \mbox{for every $j+1\leq u\leq m$}.
	\end{equation}
	Thus, to wrap up the argument it remains to compute $\theta_{u,v}$ for $1\leq u<v\leq m$ when $\{u,v\}\subset\{1,\ldots, m\}\setminus\{i,j\},$ for which one now  analyzes every possible value of $(u,v)$, according to the following cases: $$u<v<i; \:u<i<v<j; \:  u<i<j<v; \: i<u<j<v; \: j<u<v.$$
	
	Consider, e.g., the first possibility $j<u<v.$ Since $\rank M\leq 2$,  the following $3$-minor of $M$ vanishes
	$$\det\left[\begin{matrix}
		0&a_{u,i}&a_{u,j}\\
		-a_{u,v}&a_{v,i}&a_{v,j}\\
		-a_{u,i}&0&a_{i,j}\\	
	\end{matrix}\right].$$
	That is,
	$a_{u,i}(\theta_{u,v}-a_{i,j}a_{u,v})=0.$
	By hypothesis, $R$ is a domain and $a_{u,i}\neq 0.$ Thus, $\theta_{u,v}=a_{i,j}a_{u,v}.$
	
	The argument for the other listed possibilities is similar, by considering instead the matrices
	\begin{eqnarray*}
	\left[\begin{matrix}
			0&a_{u,i}&a_{u,j}\\
			-a_{u,i}&0&a_{i,j}\\
			-a_{u,v}&-a_{i,v}&a_{v,j}
		\end{matrix}\right]&,&\,\,
		\left[\begin{matrix}
			0&a_{i,u}&a_{i,j}\\
			-a_{i,u}&0&a_{u,j}\\
			-a_{i,v}&-a_{u,v}&a_{v,j}
		\end{matrix}\right], \,\,\,\,
		\left[\begin{matrix}
			0&a_{i,u}&a_{i,j}\\
			-a_{i,u}&0&a_{u,j}\\
			-a_{i,v}&-a_{u,v}&-a_{j,v}
		\end{matrix}\right]\, \quad {\rm and}\\
		 && \left[\begin{matrix}
			0&a_{i,u}&a_{i,j,}\\
			-a_{i,u}&0&-a_{j,u}\\
			-a_{i,v}&-a_{u,v}&-a_{j,v}
		\end{matrix}\right],
			\end{eqnarray*}
	respectively, to conclude that $\theta_{u,v}=a_{i,j}a_{u,v}.$
\end{proof}

For the present purpose, the relevant example  of a skew-symmetric matrix satisfying the hypotheses of Lemma~\ref{compoundskwerank2} is as follows.

\begin{Example}\label{exampleK}\rm Let $R=k[x_1,\ldots,x_n]$ be a polynomial ring in $n\geq 3$ variables over a field $k$, and let $B$ denote an $(m-2)\times m$ matrix whose entries  are polynomials in the maximal  ideal $\fm=\langle x_1,\ldots,x_n\rangle$ of $R.$ Supposing that $\Ht I_{m-2}(B)\geq 3,$  the Buchsbaum-Rim complex is a free resolution of $\coker B$ as in \ref{B-R-cokerB}. Hence, the syzygy matrix of $\coker B$ is the  $m\times m$ skew-symmetric matrix $K$ as in \eqref{KassociatedtoB}, necessarily of rank $2$.
	On the other hand, the hypothesis that $\Ht I_{m-2}(B)\geq 3$ also implies that the Eagon-Northcott complex is a minimal  free resolution of $I_{m-2}(B)R_{\fm}$ over $R_{\fm}.$  Thus, for every $1\leq i<j\leq m,$ the $i,j$th entry of $K$ is nonzero.  
\end{Example}

A major role of such symmetric matrices of rank $2$ with no-nonzero entries off the main diagonal is visible in the following result inspired by an original idea of Vasconcelos.

\begin{Proposition}\label{pre-level} Let $R=k[x_1,\ldots,x_n]$ denote a polynomial ring in $n\geq 3$ variables over a field $k$.
	Assume given the following data$:$
	\begin{enumerate}
		\item[{\rm (1)}]  An $(m+1)\times m$ block matrix $\eta:=\left[\begin{array}{c} 
				A \\ 
				\hline 
				B \end{array}\right],$
			where $A$ and $B$ are $3\times m$  and $(m-2)\times m$  matrices, respectively, with entries in $R$, with the assumption that $\Ht I_{m-2}(B)=3$, and  the entries of $B$ are polynomials belonging to the maximal  ideal $\fm=\langle x_1,\ldots,x_n\rangle$ of $R.$
		\item[{\rm (2)}]  The $m\times m$ skew-symmetric matrix $K$ in \eqref{KassociatedtoB}, with the $\Delta_{i,j}$ standing for the maximal minors of $B$.
	\end{enumerate} 
		Letting $p_1,\ldots,p_{m+1}$ denote the signed maximal minors of $\eta$  ordered in such a way that $p_1,p_2,p_3$ are those fixing the submatrix $B$, one has$:$
	\begin{enumerate}
		\item[\rm(a)] $I_2(AK)=\langle p_1,p_2,p_3\rangle I_{m-2}(B)$
		\item[\rm(b)] $I_{m-2}(B)\subset \langle p_1,p_2,p_3\rangle: I_m(\eta).$
		\item[\rm(c)] The following are equivalent$:$
		\begin{enumerate}
			\item[\rm(i)] $I_2(AK)\subset R$ has height at least two.
			\item[\rm(ii)] $\langle p_1,p_2,p_3\rangle$ has height two.
			\item[\rm(iii)] $I_{m}(\eta)\subset R$ has height two.
		\end{enumerate}
	\end{enumerate}
\end{Proposition}
\begin{proof} (a) As met previously, for every $1\leq j\leq 3$ and $1\leq u<v\leq m,$  let $h_{u,v}^{\hat{j}}$ stand for the $2$-minor of $A$ omitting the $i$th row and fixing the columns $u$  and $v.$ Expanding $p_{j}$ along of the rows of $B$ it obtains 
	\begin{equation}\label{expressaop_j}
		\sum_{1\leq u<v\leq m} h_{u,v}^{\hat{j}}\Delta_{u,v}	= p_j,\quad 1\leq j\leq 3.
	\end{equation} 
	Note that $$C_2(A)=\left[\begin{array}{ccccccccc} h_{1,2}^{\hat{3}}&h^{\hat{3}}_{1,3}&\cdots&h^{\hat{3}}_{1,m}&h^{\hat{3}}_{2,3}&\cdots& h^{\hat{3}}_{2,m}&\cdots& h^{\hat{3}}_{m-1,m}\\ [3pt]
		h_{1,2}^{\hat{2}}&h_{1,3}^{\hat{2}}&\cdots&h^{\hat{2}}_{1,m}&h^{\hat{2}}_{2,3}&\cdots& h^{\hat{2}}_{2,m}&\cdots& h^{\hat{2}}_{m-1,m}\\ [3pt]
		h_{1,2}^{\hat{1}}&h_{1,3}^{\hat{1}}&\cdots&h^{\hat{1}}_{1,m}&h^{\hat{1}}_{2,3}&\cdots& h^{\hat{1}}_{2,m}&\cdots& h^{\hat{1}}_{m-1,m}
	\end{array}\right].$$ 
	On the other hand,  apply Lemma~\ref{compoundskwerank2} with $M=K$, along  with the observation in Example~\ref{exampleK} afforded by the assumption in datum (1) to the effect that $\Ht I_{m-2}(B)=3$.
	It entails:
	\begin{equation*}
		C_2(K)=\left[\begin{matrix}\Delta_{1,2}\Delta_{1,2}&\cdots&\Delta_{i,j}\Delta_{12}&\cdots&\Delta_{m-1,m}\Delta_{1,2}\\
			\vdots&&\vdots&&\vdots\\
			\Delta_{1,2}\Delta_{ij}&\cdots&\Delta_{i,j}\Delta_{i,j}&\cdots&\Delta_{m-1,m}\Delta_{i,j}\\
			\vdots&&\vdots&&\vdots\\
			\Delta_{1,2}\Delta_{m-1,m}&\cdots&\Delta_{i,j}\Delta_{m-1,m}&\cdots&\Delta_{m-1,m}\Delta_{m-1,m}
		\end{matrix}\right].
	\end{equation*}
	Thus,

	{\small \begin{eqnarray*}
			C_2(A)C_2(K)
			=\left[\begin{array}{ccccc}
				\Delta_{1,2}\displaystyle\sum_{u,v} h_{u,v}^{\hat{3}}\Delta_{u,v}&\cdots&\Delta_{i,j}\displaystyle\sum_{u,v} h_{u,v}^{\hat{3}}\Delta_{u,v}&\cdots&\Delta_{m-1,m}\displaystyle\sum_{u,v} h_{u,v}^{\hat{3}}\Delta_{u,v}\\ [3pt]
				\Delta_{1,2}\displaystyle\sum_{u,v} h_{u,v}^{\hat{2}}\Delta_{u,v}&\cdots&\Delta_{i,j}\displaystyle\sum_{u,v} h_{u,v}^{\hat{2}}\Delta_{u,v}&\cdots&\Delta_{m-1,m}\displaystyle\sum_{u,v} h_{u,v}^{\hat{2}}\Delta_{u,v}\\ [3pt]
				\Delta_{1,2}\displaystyle\sum_{u,v} h_{u,v}^{\hat{1}}\Delta_{u,v}&\cdots&\Delta_{i,j}\displaystyle\sum_{u,v} h_{u,v}^{\hat{1}}\Delta_{u,v}&\cdots&\Delta_{m-1,m}\displaystyle\sum_{u,v} h_{u,v}^{\hat{1}}\Delta_{u,v}
			\end{array}\right].
	\end{eqnarray*}}
	Now, from \eqref{expressaop_j},
	$$C_2(A)C_2(K)=\left[\begin{matrix}
		\Delta_{1,2}\,p_3&\cdots&\Delta_{i,j}\,p_3&\cdots&\Delta_{m-1,m}\,p_3\\
		\Delta_{1,2}\,p_2&\cdots&\Delta_{i,j}\,p_2&\cdots&\Delta_{m-1,m}\,p_2\\
		\Delta_{1,2}\,p_1&\cdots&\Delta_{i,j}\,p_1&\cdots&\Delta_{m-1,m}\,p_1
	\end{matrix}\right].$$
	With this and the fact that  $C_2(AK)=C_{2}(A)C_2(K)$ we conclude that  $I_2(AK)=I_1(C_2(AK))=\langle p_1,p_2,p_3\rangle I_{m-2}(B)$ as claimed.
	
	(b) It is enough to show that $$ I_{m-2}(B) \langle p_4,\ldots,p_{m+1}\rangle\subset \langle p_1,p_2,p_3\rangle.$$ 
	For this, note that, since $[p_1\,\,\cdots\,\, p_{m+1}]\,\eta=\boldsymbol0$, then  
	$$[p_1\,p_2\,p_3]A=- [p_{4}\,\,\cdots\,\,p_{m+1}]B.$$ Thus,
	$I_1([p_{4}\,\,\cdots\,\,p_{m+1}]B)\subset\langle p_1,p_2,p_3\rangle.$ In particular, for an arbitrary $(m-2)\times (m-2)$ submatrix  $\widetilde{B}$ of $B$, one has $I_1([p_{4}\,\,\cdots\,\,p_{m+1}]\widetilde{B})\subset \langle p_1,p_2,p_3\rangle.$ Thus, $$\det \widetilde{B}\langle p_4,\ldots,p_{m+1}\rangle =I_1(\det\widetilde{B}[p_4\ \,\cdots\, p_{m+1}])=I_1([p_4\,\cdots\,p_{m+1}]\widetilde{B}\, {\rm adj}\widetilde{B}))\subset\langle p_1,p_2,p_3\rangle,$$
	where ${\rm adj}\widetilde{B}$ denotes the adjugate matrix of $\widetilde{B}$.
	
	Hence, $\det\widetilde{B}\in \langle p_1,p_2,p_3\rangle:I_{m}(\eta).$ Therefore,
$$ I_{m-2}(B) \langle p_4,\ldots,p_{m+1}\rangle\subset \langle p_1,p_2,p_3\rangle,$$
	as desired.

	(c) The implication (i)$\Rightarrow$(ii) is a consequence of (a), while  (ii)$\Rightarrow$(iii) follows from the fact that $\langle p_1,p_2,p_3\rangle$ is a subideal of $I_{m}(\eta)$. Finally, to prove the implication (iii)$\Rightarrow$(i) note by the items (a) and (b)  that  $I_m(\eta)I_{m-2}(B)^2\subset I_2(AK).$ So, if   $\Ht I_m(\eta)\geq 2$  then the height of $I_2(AK)$ is also at least $2$. 
\end{proof}

\subsection{Proof of the main theorem}

Namely, as stated in Theorem~\ref{mainthm}, our main result  shows that the ideal of maximal minors fixing a submatrix of a level matrix answers Problem~\ref{mainquestion} affirmatively, and that, in addition, every ideal of height $2$ generated by three forms  of degree $d\geq 2$ with resolution as in \eqref{res-prelim-bis} is necessarily of this form.

For the reader's convenience, we recall the statement of Theorem~\ref{mainthm}.

 \newtheorem*{Thm}{Theorem}
 \begin{Thm}  Let $J\subset R$ be an ideal.  The following conditions are equivalent$:$
	\begin{itemize}
		\item[{\rm (i)}]  $J$ is an almost Cohen--Macaulay codimension $2$ ideal generated by three forms of the same degree $d\geq 2$.
		\item[{\rm (ii)}] There exist  latent data $(d,m, \underline{\delta},\underline{\epsilon})$, and a $(d,m, \underline{\delta},\underline{\epsilon})$-level 2-block matrix $\eta$ such that $J$ is generated by the maximal minors of $\eta$ fixing its lower block consisting of $m-2$ rows.
	\end{itemize}
\end{Thm}
\begin{proof}  (ii) $\Rightarrow$ (i) This implication is a consequence of Proposition~\ref{AS-implies-ideal-prescribed-shifts} and Proposition~\ref{pre-level}.

	(i) $\Rightarrow$ (ii)
The minimal graded free resolution of $J$ is as in \eqref{res-prelim-bis}, and one sticks to the outcoming latent data $(d,m,\underline{\delta},\underline{\epsilon})$ afforded by \eqref{crescente}, \eqref{sumHS-c1}, \eqref{c2} and \eqref{c3}.
	
Thinking of $\psi$ as a matrix, set
$B:=\psi^t$.  The Buchsbaum--Eisenbud acyclicity  criterion implies that ${\rm ht\,}I_{m-2}(B)=3.$ Thus, as in the proof of Theorem~\ref{AS-implies-ideal-prescribed-shifts}, the Buchsbaum-Rim complex of $B$ is a minimal free resolution of $\coker B$, with syzygy matrix 
	$$K=\left[\begin{matrix}
		0&\sigma_{1,1}\Delta_{1,2}&\cdots&\sigma_{1,m}\Delta_{1,m}\\
		\sigma_{1,2}\Delta_{2,1}&0&\cdots&\sigma_{2,m}\Delta_{2,m}\\
		\vdots&\vdots&\ddots&\vdots\\
		\sigma_{m,1}\Delta_{m,1}&\sigma_{m,2} \Delta_{m,2}&\cdots&0
	\end{matrix}\right]$$
	as in \eqref{KassociatedtoB}. Then, dualizing \eqref{res-prelim-bis} yields $$\phi=AK,$$ for a certain
	$3\times m$ matrix $A$ such that, for every $1\leq i\leq m,$ the entries of the $i$th column are equal to zero if $d<\delta_i $, and  a homogeneous polynomial of degree $d-\delta_i$, otherwise.
	
	One now claims that the following $(m+1)\times m$ matrix
	\begin{equation}
		\eta=\left[\begin{matrix}
			A\\
			B
		\end{matrix}\right]
	\end{equation}
	satisfies the requirement in (i), namely, that
	 $J$ is generated by the three minors $p_1,p_2,p_3$ of $\eta$ fixing the submatrix $B$  of $\eta$.
	
	Note that $p_1,p_2,p_3$ are of degree $d.$  Moreover, by the Buchsbaum--Eisenbud acyclicity  criterion,  ${\rm ht\,}I_2(\phi)\geq 2.$ Thus, since $\phi=AK,$  Proposition~\ref{pre-level}(c) implies that  $\Ht I_m(\eta)=2.$ In particular, $\eta$ is a $(d,m,\underline{\delta},\underline{\epsilon})$-level matrix.   Hence, by Theorem~\ref{AS-implies-ideal-prescribed-shifts},  
	$$[p_1\,\,p_2\,\,p_3]\phi=\boldsymbol0.$$
	Say, $J=\langle f_1,f_2,f_3\rangle$.
	As $\rank \phi=2$, then $[p_1\,\,p_2\,\,p_3]$  and $[f_1\,\,f_2\,\,f_3]$ are multiples of each other in the fraction field of $R.$ But,   $[p_1\,\,p_2\,\,p_3]\neq \boldsymbol0$.  Therefore, there are nonzero elements $p,f$ of $R$  with gcd$(p,f)=1$ such that $$fp_i=pf_i,\quad 1\leq i\leq 3.$$
	Since $\deg p_i=\deg f_i=d$ for every $1\leq i\leq 3$, and $J$ has height $2$, forcibly  $p$ and $f$  are nonzero elements of $k.$ Therefore, $J=\langle p_1,p_2,p_3\rangle.$	
\end{proof}

\section{Theory guiding examples}

The examples in this section have the purpose to illustrate non-trivially the content of the main results, by gathering the precise format of the involved matrices and the shape of the related free resolutions.
One goal here is to illustrate how the search for an appropriate level matrix is often subtle.

\subsection{Non-geometric example}

The example below aims at explaining  a natural choice of a level matrix when the given ideal $J$ is itself the ideal of maximal minors fixing a submatrix.

Assume given:

$\bullet$ Latent data $(d,m, \underline{\delta},\underline{\epsilon})$  satisfying the following additional condition:  for some $1\leq u\leq m-2,$ $\delta_u=\delta_{u+1}=\delta_{u+2}=d.$

$\bullet$ 
An $(m-2)\times m$ matrix $B=(b_{i,j})$ with entries in a standard polynomial ring $R$ over a field, satisfying the condition that the $b_{i,j}$ is a homogeneous polynomial of degree $\delta_{i+2}-\delta_j+\epsilon_i$ if $\delta_{i+2}-\delta_j+\epsilon_i>0$ and $b_{i,j}=0$ if $\delta_{i+2}-\delta_j+\epsilon_i\leq 0.$

$\bullet$
	The $(m-2)\times (m-3)$ submatrix $B'$ of $B$ obtained by omitting the columns in positions $u,u+1, u+2$. 

The following proposition gives extra precision  to the content of \cite{AnSi1986}, by throwing additional light based on the  present considerations.

\begin{Proposition}\label{mainII} With the above data and notation, let $J$ be the ideal generated by the three $(m-2)$-minors of $B$ fixing the  submatrix $B'$. 
	Assume that$:$
	\begin{enumerate}
		\item[\rm(i)]
	$\Ht I_{m-2}(B)\geq 3$ and $\Ht I_{m-3}(B')\geq 2.$  
	\item[\rm(ii)] $K$ is the syzygy matrix of $\coker B$ as in \eqref{KassociatedtoB}.
		\end{enumerate}
	Then$:$
		\begin{enumerate}
			\item[\rm(a)] The minimal graded free resolution of $J$ is
			$$0\to\bigoplus  R(-d-\delta_{i+2}-\epsilon_i)\to R(-2d)^3\oplus\sum_{j\neq u,u+1,u+2} R(-d-\delta_j)^{m-3}\to  R(-d)^{3}\to R.$$
			\item[\rm(b)] The syzygy matrix of $J$ is the $3\times m$ submatrix of $K$ of  rows $u,\,u+1$ and $u+2$.
		\end{enumerate}
\end{Proposition}
\begin{proof} With no loss of generality, assume that $u=1.$ The proof for $1\leq u\leq m-2$ arbitrary is entirely similar.
	
	(a) Let $\eta$  be the following $(m+1)\times m$ matrix 
	$$\eta=
	\left[\begin{array}{c}
		\begin{matrix}\mathfrak{I}&\boldsymbol0\end{matrix}\\
		\hline
		B
	\end{array}\right]$$
	where $\mathfrak{I}$ is the $3\times 3$ identity matrix and $\boldsymbol0$ is the $3\times (m-3)$ null matrix. Note that  the data $(d,m,\underline{\delta},\underline{\epsilon})$ satisfy conditions \eqref{sumHS}, \eqref{c2} and \eqref{c3},  and the entries  are as in Definition~\ref{def_level} with respect to $(d,m,\underline{\delta},\underline{\epsilon}).$ 
	
	{\bf Claim.}  $\eta$ is a $(d,m,\underline{\delta},\underline{\epsilon})$-level matrix. 
	
	By hypothesis, $\Ht I_{m-2}(B)\geq 3.$ Thus, it remains to show that $\Ht I_{m}(\eta)=2.$ 
	But, by the format of $\eta$, the subideal of $I_m(\eta)$ generated by the maximal minors of $\eta$ fixing 	$\left[\begin{matrix}\mathfrak{I}&\boldsymbol0\end{matrix}\right]$ is exactly $I_{m-3}(B').$ Since $\Ht I_{m-3}(B')\geq 2$, then $\Ht I_{m}(\eta)\geq 2.$ 
	
	Finally, by the shape of $\eta$, $J$ is the ideal generated by the maximal minors of the $(d,m,\underline{\delta},\underline{\epsilon})$-level matrix fixing the submatrix $B.$ Hence,  Proposition~\ref{AS-implies-ideal-prescribed-shifts} wraps up the matter.
	
	(b) This is because, according to the  Proposition~\ref{AS-implies-ideal-prescribed-shifts}, the syzygy matrix of $J$ is the product 	$\left[\begin{matrix}\mathfrak{I}&\boldsymbol0\end{matrix}\right]K.$
\end{proof}

\subsection{Geometric examples}

\subsubsection{Plane curves whose Jacobian ideal admits only three generating syzygies}
	As a move toward understanding the watershed between arbitrary forms and partial derivatives of a form, one may ask whether there exists some analogue of Theorem~\ref{mainthm} in the case where $J$ is the Jacobian ideal of a form  $f \in k[x,y,z]_{d+1}$ which implies a known class of plane curves.
	As a step toward understanding this question, note that by (\cite[Theorem~2.6]{Dimca-Sticlaru2018}), drawing upon Theorem~\ref{mainthm} (i) $\Rightarrow$ (ii), we know that the Jacobian ideal of a nearly free plane curve $f \in k[x,y,z]_{d+1}$ turns out to be generated by the maximal minors of a level matrix fixing the last row.
Alas, the converse does not hold in general. 

\begin{Example}\rm The reduced plane curve defined by
$$f : = x\bigl(xy(x+y) + z^{3}\bigr),$$
is not nearly free, and yet its Jacobian ideal is generated by the maximal minors of a level matrix fixing the last row.
\end{Example}
This curve is not nearly free as noted in \cite[Remark 2.7(ii)]{Dimca-Sticlaru2018}.

On the other hand, its Jacobian ideal has the following minimal free resolution
	\begin{equation*}
	0\to R(-3-3-1)\stackrel{\psi}\longrightarrow
	R(-5)^2\oplus R(-6) \stackrel{\phi}\longrightarrow R(-3)^3\to R,
\end{equation*}
for suitable $\phi$, and
$$\psi=\left[\begin{matrix}
3/4z^2\\
y^2\\
-x-2y
\end{matrix}\right]
,$$
hence falls withing the format  \eqref{res-prelim-bis}, which by Theorem~\ref{mainthm} implies that it
is generated by the maximal minors fixing the last row of a suitable $(3,3, \underline{\delta},1)$-level matrix with $\underline{\delta}=\{2,2,3\}$.
Such a level matrix is, e.g.,
$$\eta=\left[\begin{matrix}
	0&x&0\\
	0&-3y&4\\[2pt]
	-x&-\frac{1}{3}z&0\\[2pt]
	& \quad\text{\rm transpose of} \; \psi &
\end{matrix}\right].$$ 

\begin{Example}\rm (Extended case of Dimca--Sticlaru) (char$(k)=0$) For an integer $d\geq 4$, let $f=xyzF\in R=k[x,y,z],$ where $F\in R_{d-2}$ defines a smooth hypersurface in $\mathbb{P}^2$. Assume that $\Ht \langle xF_x,yF_y,zF_z \rangle=3.$

	Note that
	$$\frac{\partial f}{\partial x}=yz(F+xF_x),\quad\frac{\partial f}{\partial y}=xz(F+yF_y),\quad\frac{\partial f}{\partial z}=xy(F+zF_z).$$
	
	Consider the following matrix $$\mathcal{N}=\left[\begin{matrix}
		x&0&0&\\
		0&y&0\\
		0&0&z\\
		F+xF_x&F+yF_y&F+zF_z
	\end{matrix}\right].$$
	
	Clearly, the partial derivatives above are the maximal minors $\mathcal{N}$ fixing the last row.
	We claim that $\mathcal{N}$ is a $(d,3 ,\underline{\delta},\underline{\epsilon})$-level matrix for the Jacobian ideal $J_f$, with $\delta_i=d-1$ for every $1\leq i\leq 3$ and  $\epsilon_1=d-2.$ Since, for these values, as one easily sees, the upper $3\times 3$ submatrix and the lowest submatrix satisfy the requisites of Definition~\ref{def_level},  it remains to show that the ideal $g:=\langle F+xF_x,F+yF_y,F+zF_z \rangle$ has height three. 
	This is clearly the case as, by the Euler relation, one has
	$$(d+1)F=3F+(d-2)F=3F+xF_x+yF_y+zF_z\in g,$$
	hence $\langle xF_x,yF_y,zF_z \rangle\subset g$ (actually an equality).
	
	Thus, according to Theorem \ref{mainthm}, the minimal graded free resolution of the Jacobian ideal $J_f$ is 
	$$0\to R(-3d+3)\to R(-2d+1)^3\to R(-d)^3\to R.$$
	
	In particular, we may take $F$ to be the Fermat form $F=x^{d-2}+y^{d-2}+z^{d-2}$ (such as in \cite[Proposition 4.2]{Dimca-SticlaruJacobianSameDeg}), or any general form of degree $d-2$ for that matter.
Actually, the assumption that $\Ht \langle xF_x,yF_y,zF_z \rangle=3$ is equivalent to requiring that pure powers of each among $x,y,z$ appear effectively in $F$. And yet, the main features of the above example are not a privilege of this assumption as there are examples of the form $f=xyzF$, with $V(F)$ smooth, for which the Jacobian ideal of $f$ illustrates Theorem~\ref{mainthm}, as  in the next piece. 
\end{Example} 

\begin{Example}\rm
	Let $f=xyz(x^d+xy^{d-1}+yz^{d-1})\in k[x,y,z]$ with $d\geq 3$. Here, $F:=x^d+xy^{d-1}+yz^{d-1}$ defines a smooth hypersurface in $\mathbb{P}^2$, but this time around $\Ht \langle xF_x,yF_y,zF_z \rangle=2.$ 
\end{Example}
One can show that the minimal free resolution of the Jacobian ideal of $f$ has the form
$$0\to R(-3d+4)\to R(-2d+2)^2\oplus R(-2d+1)\to R(-d-2)^3\to R.$$
The actual entries of the matrices $\phi$ and $\psi$ representing the differentials of the complex are involved, but an associated $(d+2, 3, \underline{\delta},1)$-level matrix has the following shape
	$$\eta=\left[\begin{matrix}
0&0&x&\\[3pt]
-xy&\frac{d(d-1)}{d-2}y^2&-\frac{d(d+1)-1}{d-2}y\\[3pt]
	xz&-2\frac{d-1}{d-2}yz&\frac{2d+1}{d-2}z\\[3pt]
	& \quad\text{\rm transpose of} \; \psi &
	\end{matrix}\right],$$
a verification left to the reader. 
Of course, the true intent is to first guess such an $\eta$ from which the shape of the free resolution above follows by Theorem~\ref{mainthm}.

\begin{Example}\rm (Higher cuspidal plane curves)
	For an integer $d\geq 2$ let $f=x^{d+1}+y^dz\in R=k[x,y,z].$ 
	The matrix
	$$\eta=\left[\begin{matrix}
		y^{d-1}&0&0\\
		0&1&0\\
		0&0&1\\
		(d+1)x^d&y&dz
	\end{matrix}\right]$$
can be seen to be $(d, \underline{\delta})$-level, with $\delta_1=1, \delta_2=\delta_3=d$.

Note that the partial derivatives of $f$ are (up to sign) the maximal minors of $\eta$ fixing the last row.	
Thus, by Theorem~\ref{mainthm} the minimal graded free resolution  
of the Jacobian ideal $J$ of $f$ is of the form
$$0\to R(-2d-1)\to R(-2d)^2\oplus R(-d-1)\to R(-d)^3\to R.$$
\end{Example}



\subsubsection{Plane curves with $4$-generated Jacobian syzygies}

In this part we point out a few examples in the case where the ideal of maximal minors fixing a submatrix is the Jacobian ideal of a form.

The first example is about arrangements of generic hypersurfaces, a theme largely explored in \cite{RTY}, of which one makes essential use here.

\begin{Example}\rm
	Let $\{f_1,\ldots,f_{m-2}\} \subset R=k[x,y,z]$ be a set of general forms of degrees $2\leq \deg f_1\leq \cdots\leq \deg f_{m-2}$, and set $f:=f_1\cdots f_{m-2}.$ Say, $m\geq 6$, for simplicity.
	
	Let $B$ be the concatenation  of the $(m-2)\times (m-3)$ matrix
	$$B'=\left[\begin{matrix}
		-f_1&0&\cdots&0\\
		0&-f_2&\cdots&0\\
		\vdots&\vdots&\ddots&\vdots\\
		0&0&\cdots&-f_{m-3}\\
		f_{m-2}&f_{m-2}&\cdots&f_{m-2}
	\end{matrix}\right].$$
	with the Jacobian matrix of $\{f_1,\ldots,f_{m-2}\}.$
\end{Example}
This matrix has been introduced in \cite{RTY} for homological purposes.

Being a mix of geometric and non-geometric situation, one appeals to  Proposition~\ref{mainII}. For this purpose, set:
$$d:=\deg f-1, \,\delta_1:=\cdots:=\delta_{m-3}:=d-1, \quad\delta_{m-2}:=\delta_{m-1}:=\delta_{m}=d,$$ $$\epsilon_i:=\deg f_i\, (1\leq i\leq m-5)\,\,\,\mbox{and}\,\,\epsilon_{i}:=\deg f_{i}-1\, (m-4\leq i\leq m-2)$$
With this one can see that the $i,j$th entry of $B$ has degree $\delta_{i+2}-\delta_j+\epsilon_i$ as in the statement of Proposition~\ref{mainII}.

Drawing upon	\cite[Proposition 3.1, and Proposition 3.4]{RTY} one has  $\Ht I_{m-2}(B)\geq 3$ and $\Ht I_{m-3}(B')\geq 2$. Let $J$ denote the Jacobian ideal of $f:=f_1\cdots f_{m-2}.$  As pointed in \cite[Theorem 3.5]{RTY}, $J$ is generated by the maximal minors of $B$ fixing the submatrix $B'.$ Thus,  by Proposition~\ref{mainII} the minimal graded free resolution of $J$ is:
\begin{equation*}
	0\lar \bigoplus_{j=1}^{m-2}R(-2d-\deg f_j+1) \lar	R(-2d)^3\oplus R(-2d+1)^{m-3} \lar R(-d)^3\lar R.
\end{equation*}

Next is an example that first appeared in \cite{JNS} as a geometric analogue of a $3$-generated ideal in $k[x,y,z]$ showed by D. Lazard to the senior author (personal communication) back in 1976.

\begin{Example}[{\cite[Example 2.6]{JNS}}]\rm 
	Let $f=(x^2-y^2)z^{d-1}-(x^{d-1}-y^{d-1})x^2-y^{d+1}\in R:=k[x,y,z], (d\geq 3)$, where $k$ is a field such that char$(k)$ does not divide $d-1$. 
	
	One can write
	$$		f_x=xp,\,\,
	f_y=yq\,\,\mbox{and}\,\,
	f_z=(d-1)(x^2-y^2)z^{d-2},
	$$
	where $$p=2y^{d-1}+2z^{d-1}-(d+1)x^{d-1}\,\, \mbox{and}\,\,  q=(d-1)x^2y^{d-3}-2z^{d-1}-(d+1)y^{d-1}.$$
	Introduce the matrices
	$$B=\left[\begin{matrix}
		0&q&-yz^{d-2}&x\\
		-p&0&-xz^{d-2}&y
	\end{matrix}\right]\quad \mbox{and}\quad B'=\left[\begin{matrix}x\\y\end{matrix}\right].$$
	
 Set:
	$$m=4, \delta_1=\delta_2=\delta_3=d, \delta_4=2d-2, \epsilon_1=d-1, \epsilon_2=1.$$
{\bf Claim.} With the above values,  $B$ and $B'$  satisfy the conditions of    Proposition~\ref{mainII}.

To see this,  first note that  the $i,j$th entry of $B$ has degree $\delta_{i+2}-\delta_j+\epsilon_i.$ 
	Since, obviously $\Ht I_1(B')=2$, it remains to prove that  $\Ht I_2(B)=3$.
	Clearly,
	${\rm ht\,}I_2(B)\leq 3.$ Now, let $P$ denote a prime ideal of $R$ containing  $I_2(B).$ In particular,  $P$ contains $J.$ 
	Thus, $\langle x,y\rangle\subset P.$ Since $pq\in P$ and $pq= 4z^{2d-2}+a,$ with $a\in \langle x,y\rangle,$ then $z\in P.$ Hence, $\langle x,y,z\rangle\subset P.$ With this, one concludes that ${\rm ht\,}P\geq 3.$ Therefore, ${\rm ht\,}P= 3,$ as claimed. 

Now, $f_x,f_y$ and $f_z/(d-1)$ are the $2$-minors of $B$ that fix $B'$. Therefore, by Proposition~\ref{mainII},  the minimal graded free resolution of $J=\langle f_x,f_y,f_z\rangle$ is
	$$0\to R(-3d+1)^{2}\to R(-2d)^{3}\oplus R(-3d+2)\to R(-d)^3\to R.$$	
	Now, here the syzygy matrix $K$ of $\coker B$ is
	$$K=\left[\begin{matrix}
		0&1/(d-1)f_z&-f_y&qxz^{d-2}\\
		-1/(d-1)f_z&0&f_x&-pyz^{d-2}\\
		f_y&-f_x&0&pq\\
		qxz^{d-2}&pyz^{d-2}&-pq&0
	\end{matrix}\right].$$
	Hence, by  Proposition~\ref{mainII}(b), the syzygy matrix of $\langle f_x,f_y,1/(d-1)f_z\rangle$ is
	$$\left[\begin{matrix}
		0&1/(d-1)f_z&-f_y&qxz^{d-2}\\
		-1/(d-1)f_z&0&f_x&-pyz^{d-2}\\
		f_y&-f_x&0&pq
	\end{matrix}\right].$$ 
	As a side note, the above minimal syzygy of standard degree $2d-2$ cannot have coordinates
	forming a regular sequence of length $3$ (\cite[Theorem 2.1]{Toh}).
	\end{Example}

	\begin{Example}\rm (Higher nodal) Let $f=(y^d-x^d)z+y^{d+1}\in R:=k[x,y,z], (d\geq 2)$, over a field $k$ of zero characteristic. Then $$f_x=x^{d-1}(-dz+(d+1)x),\quad f_y=dy^{d-1}z,\quad f_z=y^d-x^d.$$
		Introduce the matrices
		$$B=\left[\begin{matrix}
			x^{d-1}&0&dz&y\\
			y^{d-1}&dz-(d+1)x&0&x
		\end{matrix}\right]\quad \mbox{and}\quad B'=\left[\begin{matrix}x^{d-1}\\y^{d-1}\end{matrix}\right].$$
		Set: 
		$$m=4, \delta_1=2, \delta_2=\delta_3=\delta_4=d, \epsilon_1=1, \epsilon_2=d-1.$$   
	
	{\bf Claim.}  $B$ and $B'$  are as in the statement of Proposition~\ref{mainII}. 
	
	The argument is similar to the one in the previous example. Thus, first note that  the $i,j$th entry of $B$ has degree $\delta_{i+2}-\delta_j+\epsilon_i.$ Since, obviously $\Ht I_1(B')=2$, it remains to prove that  $\Ht I_2(B)=3$.
		Clearly,
		${\rm ht\,}I_2(B)\leq 3.$ Now, let $P$ denote a prime ideal of $R$ containing $I_2(B).$  Note that the Jacobian ideal $J=\langle f_x,f_y,f_z\rangle$ is the ideal of the two minors of $B$ that fix $B.$ With this, and the fact that the  $2$-minor of $B$ relative to the second and third rows of $B$ is $g=dxz$, it obtains
		$$(d+1)x^d=f_x+x^{d-2}g\in P.$$
		Thus, $x\in J.$  Consequently, this time around letting $h$ denote  the  $2$-minor of $B$ relative to the first and second rows of $B$, it obtains 	$$y^d=f_z+x^d\in P\quad \mbox{and}\quad d^2z^2=(d+1)dzx-h\in P.$$ Hence, $\langle x,y,z\rangle \subset P.$ Therefore, $\Ht P=3.$ In particular, $\Ht I_{2}(B)=3$ as was claimed.
	
	Finally, by Proposition~\ref{mainII}  the minimal graded free resolution of $J=\langle f_x,f_y,f_z\rangle$ is
		$$0\to R(-3d+1)^{2}\to R(-2d)^{3}\oplus R(-d-2)\to R(-d)^3\to R.$$
		One may observe that this resolution generalizes the one in \cite[Theorem 2.3]{SimisOrdinary}.
\end{Example}

		\bibliographystyle{amsalpha}

\begin{thebibliography}{10}
			
			\bibitem{Miro-et-al} A. Andrade, V. Beorchia and R. M. Miró-Roig, A characterization of quasi-homogeneous singularities of free and nearly free plane curves,  Int. Math. Res. Notices, {\bf 1} 2025(1), 1--21.
			
			\bibitem{AnSi1981} J. Andrade and A. Simis, A complex that resolves the ideal of minors having $n-1$ columns in common,            Proc. Amer. Math. Soc. {\bf 81} (1981), 217--219.
			
			\bibitem{AnSi1986}{J. Andrade and A. Simis, Ideals of minors fixing a submatrix, J. Algebra {\bf 102} (1986),  	249--259.}
			
			\bibitem{Ellia1998} V. Beorchia and Ph. Ellia, On the equations defining quasi complete intersection space curves, Arch. Math. {\bf 70} (1998), 244--249.
		
			
			\bibitem{Bruns} W. Bruns, “Jede” endliche freie Auflösung ist freie Auflösung eines von drei Elementen erzeugten Ideals, J. Algebra {\bf 39}  (1976), 429--439. 
			
			
				\bibitem{BHbook}{W. Bruns, J. Herzog, Cohen–Macaulay Rings, Cambridge University Press, Cambridge, 1993.}
				
				\bibitem{Burch} L. Burch, A note on the homology of ideals generated by three elements in local
				rings, Proc. Cambridge Philos. Soc. {\bf 64} (1968), 949--952.
			
			\bibitem{RTY} R. Burity, Z. Ramos, A. Simis, and S. Toh\v{a}neanu, Rose-Terao-Yuzvinsky theorem for reduced
			forms, J. Algebra {\bf 673} (2025), 45–76.
			
		
		
		
		
			
	\bibitem{Dimca-Sticlaru2018} A. Dimca and G. Sticlaru, Nearly free divisors vs. rational cuspidal curves, Publ. RIMS Kyoto Univ.  {\bf 54} (2018), 163--179.

\bibitem{Dimca-Sticlaru2019} A. Dimca and G. Sticlaru, Saturation of Jacobian ideals: Some applications to nearly free curves, line arrangements and rational cuspidal plane curves, J. Pure Applied Algebra
{\bf 223} (2019), 5055--5066.

		
		
			\bibitem{Dimca_Sticlaru-Geom.Ded} A. Dimca and G. Sticlaru, Plane curves with three syzygies, minimal Tjurina curves curves,
		and nearly cuspidal curves. Geom. Dedicata {\bf 207} (2020) 29--49.
		
		
		\bibitem{Dimca-SticlaruJacobianSameDeg} A. Dimca and G. Sticlaru, Curves with Jacobian syzygies of the same degree, Rendiconti del Circolo Matematico di Palermo Series 2 (2025) 74:84.
		
	\bibitem{Dimca-SticlaruJacobianDeriv} A. Dimca and G. Sticlaru, Bourbaki modules and the module of Jacobian derivations of projective hypersurfaces, arXiv:2506.23950v3 [math.AG] 14 Jul 2025.
		
			
		
		\bibitem{E}{D. Eisenbud, Commutative algebra with a view toward algebraic geometry,  vol. 150, Graduate Texts in Mathematics, Springer-Verlag, New York, 1995.}
		
		\bibitem{PPPideals} G. Gama, D. Lira, Z. Ramos and A. Simis, Cohen–Macaulay ideals of codimension two and the geometry of plane points, Bull Braz Math Soc, New Series, Vol. 56, {\bf 60} (2025). https://doi.org/10.1007/s00574-025-00488-x
		
	\bibitem{Gulliksen}	T. Gulliksen, Tout idéal premier d'un anneau noethérien est associé à un idéal engendré par trois éléments, C. R. Acad. Sei. Paris Sér. A-B {\bf 271} (1970), A1206--A1207.
		
		\bibitem{HS}{S. H. Hassanzadeh, A. Simis, Plane Cremona maps: Saturation and regularity of the base ideal, J. Algebra {\bf 371} (2012), 620–652}
		
				
		
		\bibitem{JNS}{M. Jardim, A. N. Nejad and A. Simis,  The Bourbaki Degree of plane projective curves, Trans. Amer. Math. Soc. (2024),  7633-7655.}
			
	\bibitem{Kohn}	P. Kohn, Ideals generated by three elements, Proc. Amer. Math. Soc. {\bf 35} (1972), 55--58.
			
			\bibitem{Lazard} D. Lazard,  Alg\`ebre lin\'eaire sur $K[x_1, \ldots, x_n]$ et \'elimination, Bulletin de la S. M. F., Tome 105 (1977), 165--190.
			
			\bibitem{Ma_et_al2021} G. Malara, P. Pokora and H. Tutaj-Gasinska, On $3$-syzygy and unexpected plane curves, Geom. Dedicata, {\bf 214} (2021), 49--63.
			
			
			\bibitem{McCu2011} J.  McCullough, A family of ideals with few generators in low degree and large projective dimension,
			Proc. Amer. Math. Soc. {\bf 139} (2011), 2017--2023.
			
			\bibitem{McCu} J.  McCullough, Prime Ideals and Three-generated Ideals with Large Regularity, Comptes Rendus. Mathématique, Volume {\bf 362} (2024),  251--255.
		
			
				
				
					
					
					\bibitem{SimisOrdinary} A. Simis, A note on the Koszul homology of ordinary singularities, Bol. Soc. Bras. Mat. {\bf 8.2} (1977), 149--159.
					
				
				\bibitem{SimisTohaneanu} A. Simis and S. O. Toh\v aneanu,	Homology of homogeneous divisors. Isr. J. Math. {\bf 200} , (2014) 449--487.
				
			
	\bibitem{Toh} S. O. Toh\u aneanu, On Freeness of Divisors on $\pp^2$, Comm. in Algebra, {\bf 41} (2013), 2916--2932. 
				
		\end{thebibliography}

\pagebreak		
		
	\noindent	
	\begin{minipage}{7cm}
		\begin{center}	
			{\sc Ricardo Burity} \\          
			Departamento de Matem\'atica, CCEN\\
			Universidade Federal da Para\'iba \\
			58051-900 J, Pessoa, PB, Brazil\\
			ricardo@mat.ufpb.br \hspace{5cm}
		\end{center}
	\end{minipage} \quad 
	\begin{minipage}{7cm}
		\begin{center}	
			\noindent {\sc Thiago Fiel}\\
			Departamento de Matem\'atica, CCEN\\ 
			Universidade Federal de Pernambuco\\ 
			50740-560 Recife, PE, Brazil\\
			thiagofieldacostacabral@gmail.com
		\end{center}
	\end{minipage} 
	
	\bigskip
	
	\noindent	
	\begin{minipage}{7cm}
		\begin{center}	
			\noindent {\sc Zaqueu Ramos}\\
			Departamento de Matem\'atica, CCET\\ 
			Universidade Federal de Sergipe\\
			49100-000 S\~ao Cristov\~ao, SE, Brazil\\
			zaqueu@mat.ufs.br
		\end{center}
	\end{minipage}
	\quad 
	\begin{minipage}{7cm}
		\begin{center}
			\noindent {\sc Aron Simis}\\
			Departamento de Matem\'atica, CCEN\\ 
			Universidade Federal de Pernambuco\\ 
			50740-560 Recife, PE, Brazil\\
			aron.simis@ufpe.br
		\end{center}
	\end{minipage}
		
		\end{document}